\documentclass[11pt,reqno]{amsart}
\usepackage{amsmath,amsfonts,amssymb,amscd,amsthm,amsbsy,epsf}
\newtheorem*{Thm}{Theorem}
\newtheorem{thm}{Theorem}
\newtheorem*{cor}{Corollary}
\newtheorem*{lem}{Lemma}
\theoremstyle{definition}
\newtheorem*{rem}{Remark}
\newtheorem*{rems}{Remarks}
\def\ep{\varepsilon}
\def\F{{\mathcal F}}
\def\S{{\mathcal S}}
\def\R{{\mathcal R}}
\def\real{{\mathbb R}}
\def\whf{\widehat f\,}
\begin{document}
\title
[Multilinear Embedding Estimates for the Fractional Laplacian]
{Multilinear Embedding Estimates for\\ the Fractional Laplacian}
\author{William Beckner}
\address{Department of Mathematics, The University of Texas at Austin,
1 University Station C1200, Austin TX 78712-0257 USA}
\email{beckner@math.utexas.edu}
\begin{abstract}
Three novel multilinear embedding estimates for the fractional Laplacian are 
obtained in terms of trace integrals restricted to the diagonal.
The resulting sharp inequalities may be viewed as extensions of the 
Hardy-Littlewood-Sobolev inequality, the Gagliardo-Nirenberg 
inequality and Pitt's inequality.
\end{abstract}
\maketitle

Sobolev embedding estimates are a central tool for analysis on geometric 
manifolds. 
Natural questions arise with the study of multilinear operators and product 
manifolds that incorporate intrinsic geometric symmetry. 
New realizations for the fractional Laplacian have emerged as critical 
elements for resolving challenging issues in nonlinear analysis and 
conformal geometry.
Development of a rigorous framework for central problems in mathematical 
physics, including the structure and stability of matter and the 
dynamics of many-body interaction, has suggested new applications for 
estimates that measure fractional smoothness.
Direct methods of approach have proved highly successful, but 
determining intrinsic connections with the overall framework of Sobolev 
embedding and fractional integrals is important and useful to gain new 
insight and increased understanding for the analytical structure. 
The effort to calculate optimal constants for embedding estimates and 
convolution integrals underlines not only intrinsic features for exact model
problems and encoded geometric information, but lays the groundwork for 
calculating precise lower-order effects 
(see \cite{Beckner-Forum}, \cite{FLS-JAMS}). 
Motivated by current interest to model the many-body dynamics of a 
Bose gas using the Gross-Pitaevskii hierarchy of density matrices, 
three new results are given for multilinear embeddings for the fractional 
Laplacian on $\real^n$ that can be viewed as separate extensions of the 
Hardy-Littlewood-Sobolev inequality, the Gagliardo-Nirenberg inequality 
and Pitt's inequality for the Fourier transform with weights. 
The simplicity of the argument underscores both the naturalness and the 
novelty of the result. 
The context for these inequalities has two explicit themes:
(1)~the development of the Gross-Pitaevskii hierarchy to describe 
multi-particle dynamics requires control of multilinear smoothing 
estimates where the evident natural question is to determine the analytic 
control that the smoothing estimates give in terms of physical measures 
such as trace integrals (see the operators $S_j$ and $R_j$ and the 
iterated multiplier argument in \cite{KM}); and \eqref{eq2} 
to determine the behavior of convolution integrals and Young's inequality 
for multilinear product decomposition of the manifold in the context of 
Riesz potentials and Stein-Weiss fractional integrals.

Consider the convolution integral with multi-component decomposition 
$$F(w) =\int_{\real^n\times\cdots\times\real^n} 
G(w-y)\, H(y)\,dy\ ,\qquad w\in \real^{mn}$$
\begin{itemize}
\item[(1)] ``diagonal trace restriction''
$$F(w) \rightsquigarrow F(\underbrace{x,\ldots,x}_{m\text{ slots}}) 
\equiv F(x)\ ,\qquad x\in \real^n$$
\item[(2)] ``multilinear products''
$$F(x) = \int_{\real^{mn}} \prod g_k (x-y_k)\, H(y_1,\ldots,y_m)\, dy\ ;$$
\end{itemize}
the objective is to determine how the components of $G$ and the 
nature of $H$ control the size of $F$.
An important point to emphasize initially is that only for special cases 
will the analysis reduce to an iterative or product function characterization.
Here the $g_k$'s will be taken as inputs, including Riesz potentials, so 
the multilinear map is given by 
$$H\in L^p (\real^{mn}) \rightsquigarrow F\in L^q (\real^n)\ .$$

A relatively simple lemma that characterizes this framework can be easily 
obtained from the classical Young's inequality.

\begin{lem} 
For $1\le p,q<\infty$ with $1<s_k < p'$ and $m/p' + 1/q = \sum 1/s_k$ 
(primes denote dual exponents, $1/p + 1/p' =1$)
$$\|F\|_{L^q(\real^n)} \le C\prod \|g_k\|_{L^{s_k}(\real^n)}  
\|H\|_{L^p (\real^{mn})}$$
\end{lem} 

\begin{proof} 
Consider the dual problem 
$$\bigg[ \int \Big| \int \prod_k g_k (x-y_k)\, h(x)\,dx\Big|^{p'}\,dy
\bigg]^{1/p'} 
\le C \prod \|g_k\|_{L^{s_k}(\real^n)} \|h\|_{L^{q'}(\real^n)}$$
Choose the sequence $\beta_k = q/s_k - q/p'$ with $\sum \beta_k=1$ 
and apply H\"older's inequality on the left-hand side.
\end{proof}

\begin{rems}
(1)~Lorentz-space extensions follow using interpolation and 
Hardy-Little-wood-Sobolev arguments.
(2)~In the special case where $H$ is radial decreasing and so bounded above 
by a multiple of $|y|^{-mn/p}$ which allows a splitting of $H$ by 
non-uniform inverse powers of $|y_k|$, Kenig and Stein \cite{KS-99} 
show that the allowed range of Lebesgue exponents is full for the 
indices $s_k$ and that the index $q$ can go below one.
This latter result is not possible for the general case.
(See their Lemma~7 on page~7 of \cite{KS-99}.)
(3)~The classical notion of trace is expanded here to include any integral 
which is calculated over the diagonal restriction for variables.
(4)~In the special case where $p=q=2$, the optimal value for the 
constant $C$  is 
$$\int_{\real^n} \prod (g_k * g_k) (x)\,dx\ .$$
(5)~More general versions for multilinear fractional integral kernels 
are treated in Christ \cite{Christ-TAMS}, and in the conformally 
invariant setting by Beckner \cite{Beckner-PUP95} (see Theorem~6 on 
pages~48--49). 
(6)~A formative treatment for rigorously describing dynamical processes with
many-body interaction in macroscopic systems appears in Spohn \cite{Spohn-80}.
(7)~Trace integrals have an intrinsic analytic character which allows facility in 
making exact calculations. 
Physical motivation for analytic adaptation of the trace integral is described 
in \cite{Bardos} and \cite{Spohn-80}.
But the results described here are natural since fractional smoothness determines
that restriction to a linear sub-variety is well defined (see chapter~6 in \cite{Stein70}, 
Theorem~11 in \cite{Calderon},  the main theorem in \cite{Stein62} and 
discussion on page 32 in \cite{Mazya85}).
Implicit recognition of this structure underlies one argument in \cite{KM}. 
\end{rems}

The square-integrable paradigm that represents the motivating step for 
the arguments developed here is the following representation for the 
Hardy-Littlewood-Sobolev inequality:

\begin{lem}
For $f\in \S (\real^n)$, $1<p<2$ and $\alpha = n (1/p - 1/2)$ 
\begin{gather*}
\int_{\real^n} |f|^2\,dx \le C_p \bigg[ \int_{\real^n} 
|(-\Delta/4\pi^2)^{\alpha/2} f|^p \,dx\bigg]^{2/p}\\
\noalign{\vskip6pt}
C_p = \pi^{n/p - n/2} \Big[ \Gamma (n/p')/\Gamma (n/p)\Big]
\Big[ \Gamma (n)/\Gamma (n/2)\Big]^{2/p-1}
\end{gather*}
\end{lem}

\noindent 
This lemma is an equivalent formulation using the fractional Laplacian for 
the sharp Hardy-Littlewood-Sobolev inequality calculated by Lieb \cite{Lieb}. 

Consider $m$ copies of $\real^n$ and let $f$ be in the Schwartz class 
$\S(\real^{mn})$.
Define
\begin{equation*}
(\F f) (x) = \whf (x) = \int e^{2\pi ixy} f(y)\, dy\ .
\end{equation*}
Observe that on $\real^n$ with $0<\lambda < n$
\begin{equation*}
\F\big[ |x|^{-\lambda}\big] = \pi^{-n/2 \,+\, \lambda} 
\left[\frac{\Gamma \big(\frac{n-\lambda}2\big)}{\Gamma\big(\frac{\lambda}2
\big)}\right] \, |x|^{-(n-\lambda)}\ .
\end{equation*}
For $f\in \S(\real^{mn})$, $\Delta_k=$ standard Laplacian on $\real^n$ 
in the variable $x_k$, $0<\alpha_k < n$, $\alpha = \sum \alpha_k$ for 
$k=1$ to $m$ and $(m-1)n < \alpha < mn$, define 
\begin{align*}
\Lambda (f;\alpha_1,\ldots,\alpha_m) 
& = \int_{\real^n\times\cdots\times\real^n} 
\Big| \prod_{k=1}^m (-\Delta_k/4\pi^2)^{\alpha_k/4} f\Big|^2\, 
dx_1 \ldots dx_m\\
\noalign{\vskip6pt}
& = \int_{\real^n\times\cdots\times\real^n} 
\prod_{k=1}^m |\xi_k|^{\alpha_k} |\whf|^2\, 
d\xi_1\ldots d\xi_m\ .
\end{align*}

\begin{thm}[Pitt's inequality]\label{thm1}
For $f\in \S(\real^{mn})$ and $n-\beta = mn-\alpha$
\begin{equation}\label{eq1}
\int_{\real^n} |x|^{-\beta} 
|f(\,\underbrace{x,\ldots,x}_{\text{$m$ slots}}\,)|^2\,dx 
\le C_\beta \ \Lambda (f;\alpha_1,\ldots,\alpha_m)
\end{equation}
\begin{equation*}
C_\beta = \pi^{-(m-1)n/2 +\alpha} 
\prod_{k=1}^m 
\left[\frac{\Gamma\big(\frac{n-\alpha_k}2\big)}{\Gamma\big(\frac{\alpha_k}2
\big)}\right]
\left[\frac{\Gamma\big(\frac{\beta}2\big)}{\Gamma\big(\frac{n-\beta}2\big)}
\right]
\left[\frac{\Gamma\big(\frac{n-\beta}4\big)}{\Gamma\big(\frac{n+\beta}4\big)}
\right]^2
\end{equation*}
The constant $C_\beta$ is sharp and no extremals exist for this inequality.
\end{thm}

\begin{thm}[Hardy-Littlewood-Sobolev inequality]\label{thm2}
For $f\in \S(\real^{mn})$ and $mn-\alpha = 2n/q$
\begin{equation}\label{eq2}
\left[\int_{\real^n} |f(\,\underbrace{x,\ldots,x}_{\text{$m$ slots}}\,)|^q\,dx
\right]^{2/q} 
\le F_\alpha \ \Lambda (f;\alpha_1,\ldots,\alpha_m)
\end{equation}
\begin{equation*}
F_\alpha = \pi^{\alpha/2} \prod_{k=1}^m 
\left[\frac{\Gamma\big(\frac{n-\alpha_k}2\big)}{\Gamma\big(\frac{\alpha_k}2\big)
}\right]
\left[\frac{\Gamma\big(\frac{\alpha-(m-1)n}2\big)}
{\Gamma\big(n-\frac{mn-\alpha}2 \big)}\right] 
\left[\frac{\Gamma (n)}{\Gamma \big(\frac{n}2\big)}\right]
^{\frac{\alpha-(m-1)n}n}
\end{equation*}
The constant $F_\alpha$ is sharp and extremals are given by 
\begin{equation*}
f(x_1,\ldots,x_m) 
= \int_{\real^n} \prod_{k=1}^m |x_k -w|^{-(n-\alpha_k/2)}
|1 + w^2|^{-n/q}\,dw
\end{equation*}
up to conformal automorphism of the factor $|1+w^2|^{-n/q}$.
\end{thm}

\begin{rem}
If $m=1$, the sharp forms of the classical Pitt's inequality and the 
Hardy-Littlewood-Sobolev inequality are recovered.
By choosing $f$ to be a product function, special cases of the general 
Pitt's inequality and the Stein-Weiss theorem can be obtained without 
sharp constants (see Appendix in \cite{Beckner-PAMS08}).
For both inequalities, the term on the left-hand side can be viewed as 
``restriction to the diagonal.'' 
Simple iteration allows extension of these estimates to diagonal restriction
traces on subblocks of the manifold $\real^{mn}$.
\end{rem}

\begin{proof}[Proof of Theorem~\ref{thm1}] 
Inequality \eqref{eq1} is equivalent to the multilinear fractional integral 
inequality: 
\begin{gather*}
\int_{\real^n} \Big| \int_{\real^{mn}} \prod_{k=1}^m |x-y_k|^{-(n-\alpha_k/2)}
f(y_1,\ldots,y_m)\,dy\Big|^2\ |x|^{-\beta}\,dx\\
\noalign{\vskip6pt}
\hskip1in \le D_\beta \int_{\real^{mn}} |f(x_1,\ldots,x_m)|^2\,dx\\
\noalign{\vskip6pt}
C_\beta = \pi^{-mn+\alpha} \prod_{k=1}^m 
\left[ \frac{\Gamma \big(\frac{2n-\alpha_k}4\big)}{\Gamma\big(\frac{\alpha_k}4
\big)}\right]^2\ D_\beta\ .
\end{gather*}
By $L^2$ duality this is equivalent to:
\begin{equation*}
\int_{\real^{mn}} \Big| \int_{\real^n} \prod_{k=1}^m |y_k-x|^{-(n-\alpha_k/2)}
|x|^{-\beta/2} g(x)\,dx\Big|^2\,dy 
\le D_\beta \int_{\real^n} |g(x)|^2\,dx\ .
\end{equation*}
Using rearrangement  arguments this inequality is reduced to non-negative 
radial decreasing functions $g(x)$.
The left-hand side becomes 
\begin{equation*}
\int_{\real^n\times\real^n\times\real^{mn}} \mkern-24mu
g(x)|x|^{-\beta/2} \prod_{k=1}^m 
|y_k-x|^{-(n-\alpha_k/2)|} \prod_{k=1}^m |y_k-w|^{-(n-\alpha_k/2)} 
|w|^{-\beta/2} g(w)dx\,dw\,dy .
\end{equation*}
Integrating out the $y_k$ variables 
\begin{gather*}
\int_{\real^n\times\real^n} g(x) |x|^{-\beta/2} |x-w|^{-mn+\alpha} 
|w|^{-\beta/2} g(w)\,dx\,dw 
\le E_\beta \int_{\real^n} |g(x)|^2\,dx\\
\noalign{\vskip6pt}
C_\beta = \pi^{-mn/2+\alpha} \prod_{k=1}^m 
\Gamma \Big(\frac{n-\alpha_k}2\Big)\Big/ \Gamma\Big(\frac{\alpha_k}2\Big)
\ E_\beta\ .
\end{gather*}
Since $mn - \alpha = n-\beta$, this becomes the classical Stein-Weiss 
fractional integral: 
\begin{equation*}
\int_{\real^n\times \real^n} g(x) |x|^{-\beta/2} |x-w|^{-(n-\beta)} 
|w|^{-\beta/2} g(w)\, dx\,dw 
\le E_\beta \int_{\real^n} |g(x)|^2\,dx
\end{equation*}
with 
\begin{equation*}
E_\beta = \pi^{n/2} 
\left[\frac{\Gamma \big(\frac{\beta}2\big)}{\Gamma\big(\frac{n-\beta}2\big)}
\right] 
\left[\frac{\Gamma\big(\frac{n-\beta}4\big)}{\Gamma\big(\frac{n+\beta}4\big)}
\right]^2\ .
\end{equation*}
See Theorem~3 in \cite{Beckner-Forum08} and also \cite{Beckner-PAMS95}.
Then 
\begin{equation*}
C_\beta= \pi^{-(m-1)n+\alpha} \prod_{k=1}^m 
\frac{\Gamma\big(\frac{n-\alpha_k}2\big)}{\Gamma\big(\frac{\alpha_k}2\big)}
\left[\frac{\Gamma\big(\frac{\beta}2\big)}{\Gamma\big(\frac{n-\beta}2\big)}
\right]
\left[\frac{\Gamma\big(\frac{n-\beta}4\big)}{\Gamma\big(\frac{n+\beta}4\big)}
\right]^2
\end{equation*}
with $mn-\alpha = n-\beta$.
\renewcommand{\qed}{}
\end{proof}

\begin{rems}
For notation, the Lebesgue measure $dx$ incorporates the dimension of the 
underlying domain. 
Observe that as $\beta\to0$ the constant $C_\beta$ is unbounded so that 
the requirement $\beta >0$ is strict. 
This reflects that the multilinear estimate is fully at the $L^2$ 
spectral level where one would not expect homogeneous Sobolev embedding 
without weights. 
The constant $C_\beta$ is sharp and no extremals exist which follows from 
reduction to the one variable case in $\real^n$. 
Note that if $f$ is a product function, for example $f(x) = \prod u(x_k)$, 
then the inequality reduces to the case of fractional Sobolev embedding
on $\real^n$ where the index is an even integer and one can set $\beta=0$.
But the calculation here provides no information on the constant in that case.
For large $m$, some $\alpha_k$ must approach $n$ and so again 
the constant will be unbounded. 
Iterative methods are not effective which indicates that the results are 
clearly multidimensional. 
The appearance of the factors $\Gamma (\alpha_k/2)$ in the denominator 
of the constant $C_\beta$ raises the question of how this constant will 
behave as one of the $\alpha_k$'s goes to zero. 
Observe that since $\alpha_k = \sum' (n-\alpha_\ell)+\beta$ is represented 
as a sum of positive terms with the sum taken for $\ell\ne k$, each term 
must also approach zero and in such case $C_\beta\to\infty$. 
\end{rems}

\begin{proof}[Proof of Theorem~\ref{thm2}] 
Inequality \eqref{eq2} is equivalent to the multilinear fractional integral 
inequality: 
\begin{gather*}
\bigg[ \int_{\real^n} \Big| \int_{\real^{mn}} \prod_{k=1}^m 
|x-y_k|^{-(n-\alpha_k/2)} f(y_1,\ldots,y_m) \,dy\Big|^q dx\bigg]^{2/q}\\
\le G_\alpha \int_{\real^{mn}} |f(x_1,\ldots,x_m)|^2\, dx\\
F_\alpha = \pi^{-mn+\alpha} \prod_{k=1}^m 
\bigg[ \frac{\Gamma (\frac{2n-\alpha_k}4)}{\Gamma (\frac{\alpha_k}4)}
\bigg]^2 G_\alpha
\end{gather*}
By duality this is equivalent to: 
\begin{equation*}
\int_{\real^{mn}} \Big| \int_{\real^n} \prod_{k=1}^m 
|y_k -x|^{-(n-\alpha_k/2)} g(x)\,dx\Big|^2\,dy 
\le G_\alpha\bigg[ \int_{\real^n} |g(x)|^p\,dx\bigg]^{2/p}
\end{equation*}
where $\frac1p + \frac1q =1$, $1<p<2$ and $mn-\alpha = 2n/q$.
As with the calculation for Theorem~\ref{thm1}, the left-hand side becomes
\begin{equation*}
\int_{\real^n\times \real^n\times\real^{mn}} \mkern-18mu
g(x) \prod_{k=1}^m |y_k -x|^{-(n-\alpha_k/2)} 
\prod_{k=1}^m |y_k-w|^{-(n-\alpha_k/2)} g(w)\, dx\,dw\, dy
\end{equation*}
Integrating out the $y_k$ variables 
\begin{gather*}
\int_{\real^n \times \real^n} g(x) |x-w|^{-mn+\alpha} g(w)\,dx\,dw 
\le H_\alpha \bigg[ \int_{\real^n} |g(x)|^p\,dx\bigg]^{2/p} \\
F_\alpha = \pi^{-mn/2 +\alpha} \prod_{k=1}^m 
\Gamma\left(\frac{n-\alpha_k}2\right) \Big/ 
\Gamma\left(\frac{\alpha_k}2\right) H_\alpha
\end{gather*}
Since $mn-\alpha = 2n/q$, this becomes the classical 
Hardy-Littlewood-Sobolev inequality:
\begin{equation*}
\int_{\real^n\times\real^n} g(x) |x-w|^{-2n/q} g(w)\,dx\,dw 
\le H_\alpha\bigg[ \int_{\real^n} |g(x)|^p \,dx\bigg]^{2/p}
\end{equation*}
with 
\begin{equation*}
H_\alpha = \pi^{n/\alpha} 
\frac{\Gamma \Big(\frac{n}p - \frac{n}2\Big)}{\Gamma\Big(\frac{n}p\Big)} 
\left[\frac{\Gamma (n)}{\Gamma \Big(\frac{n}p\Big)} \right]^{2/p-1}
\end{equation*}
Then 
\begin{equation*}
F_\alpha = \pi^{\alpha/2} \prod_{k=1}^m
\left[\frac{\Gamma\big(\frac{n-\alpha_k}2\big)}{\Gamma\big(\frac{\alpha_k}2\big)
}\right]
\left[\frac{\Gamma\big(\frac{\alpha-(m-1)n}2\big)}
{\Gamma\big(n-\frac{mn-\alpha}2 \big)}\right]
\left[\frac{\Gamma (n)}{\Gamma \big(\frac{n}2\big)}\right]
^{\frac{\alpha-(m-1)n}n}\ .
\end{equation*}
Extremal functions are determined by the classical inequality.
\end{proof}

The Stein-Weiss lemma (see Appendix in \cite{Beckner-PAMS08}) allows the trace 
inequality of Theorem~\ref{thm1} to be formulated in more general terms:

\begin{thm}[Stein-Weiss trace]\label{thm3} 
For $f\in\S(\real^{mn})$ consider 
\begin{equation*} 
F(x) = |x|^{-\beta/2} \int_{\real^{mn}} \prod_{k=1}^m K_k (x,y_k) 
f(y_1,\ldots,y_m)\,dy_1 ,\ldots, dy_m
\end{equation*}
where $\{K_k(x,y)\}$ is a family of non-negative kernels defined 
on $\real^n\times\real^n$, each kernel being continuous on any domain 
that excludes the diagonal, homogeneous of degree $-\sigma_k$ 
${(0< \sigma_k <n)}$, $K_k (\delta u,\delta v) = \delta^{-\sigma_k} K_k (u,v)$, 
and $K_k (Ru,Rv) = K_k(u,v)$ of any $R\in SO(n)$; $0<\beta <n$ with 
$2\sigma+\beta - mn = n$ where $\sigma =\sum \sigma_k$, 
Then 
\begin{gather}
\int_{\real^n} |F(x)|^2\,dx \le A_\sigma \int_{\real^{mn}} 
|f(y_1,\ldots, y_m)|^2\,dy \label{eq3}\\
\noalign{\vskip6pt}
A_\sigma = \int_{\real^n} |x|^{-\frac{\beta}2 - \frac{n}2} 
\prod_k \bigg[ \int_{\real^n} K_k (x,y) K_k (\xi_1,y)\,dy\bigg]\,dx
\notag
\end{gather} 
with $\xi_1$ a unit vector in the first coordinate direction
\end{thm}

\begin{proof} 
Apply the argument used for the proof of Theorem~\ref{thm1} and observe 
that the kernel 
\begin{equation*} 
\widehat K (x,w) = |x|^{-\beta/2} |w|^{-\beta/2} 
\prod_k \int_{\real^n} K_k (x,y) K_k (w,y)\,dy
\end{equation*}
satisfies the requirements of the Stein-Weiss lemma. 
The requirement for this trace estimate to hold is that $A_\sigma$ 
is finite. 
\end{proof}

In looking to understand how fractional smoothness controls size at the spectral 
level, and taking into account the dual representation given by the Fourier 
transform in balancing differentiability versus decay at infinity, the Stein-Weiss 
integral expresses a realization of the {\em uncertainty principle} 
\begin{equation}
c\int_{\real^n} |f|^2\,dx 
\le \int_{\real^n} |(-\Delta/4\pi^2)^{\alpha/4} |x|^{\alpha/2} f(x) |^2\,dx\ .
\end{equation}
An asymptotic argument gives directly the classical inequality. 
More broadly, this principle extends to include restriction to a $k$-dimensional 
linear sub-variety
\begin{equation}\label{eq5bis}
d\int_{\real^k} |\R f|^2\,dx 
\le \int_{\real^n} \Big| (-\Delta/4\pi^2)^{\alpha/4} |x|^{\beta/2} f(x)\Big|^2\,dx
\end{equation}
where $n-\alpha = k-\beta$, $n\ge k>\beta >0$ and 
$$d= \pi^{-\alpha}\ \frac{\Gamma(\frac{\alpha}2)}{\Gamma(\frac{\beta}2)} 
\left[\frac{\Gamma (\frac{k+\beta}4)}{\Gamma(\frac{k-\beta}4)}\right]^2\ .$$
Using the principle of the Stein-Weiss trace formulated above, multilinear trace 
integral embedding estimates described here can be extended to include 
iterated multiplication of fractional powers with successive alternation between 
the function side and the Fourier transform side where the optimal constants will be
given as closed-form integrals. 
Though the integral kernel does not have translation invariance on $\real^n$, 
dilation invariance transforms the problem to repeated convolution 
integrals on the multiplicative group $\real_+$ (see for example the 
section on iterated Stein-Weiss integrals in \cite{Beckner-Forum08}) 
which facilitates the closed-form computation of sharp constants.

To illustrate this framework and outline the strategy needed to treat iterated 
multilinear embedding forms 
$$\int_{\real^n\times\cdots\times \real^n} \Big| \prod |x_k|^{\rho_k/2} 
\prod_j \Big[ \prod_k (-\Delta_k/4)^{\alpha_{jk}/4} |x_k|^{\beta_{jk}/2}\Big] f\Big|^2 
dx_1\cdots dx_m$$
the following theorem includes the critical steps:

\begin{thm}[iterated Stein-Weiss]
\label{iteratedSW}
For $f\in \S(\real^{mn})$, $0<\alpha_k <n$, $\alpha = \sum \alpha_k$, 
$0<\beta_k <n$, $\beta = \sum \beta_k$, $0<\rho_k <n$, $\rho = \sum \rho_k$ 
and $n-\beta -\rho = mn-\alpha$ with $0<\beta+\rho <n$
\begin{gather*}
\int_{\real^n}\! |f(\underbrace{x,\cdots,x}_{\text{$m$ slots}})|^2\,dx 
\le C\!  \int_{\real^n\times\cdots\times\real^n} 
\Big| \prod_{k=1}^m |x_k|^{\rho_k/2} (-\Delta_k/4\pi)^{\alpha_k/4} |x_k|^{\beta_k/2} 
f\Big|^2 dx_1\cdots dx_m\\
\noalign{\vskip6pt}
C = \pi^{-mn+\alpha} \prod_{k=1}^m 
\bigg[ \frac{\Gamma (\frac{2n-\alpha_k}4)}{\Gamma (\frac{\alpha_k}4)}\bigg]^2
\ \frac{2^{-mn+\alpha/2}}{\sigma(S^{n-1})}\ \int_\real H(x)\,dx
\end{gather*}
with $H$ defined in the proof below.
The constant $C$ is sharp and no extremals exist for this inequality.
\end{thm}

\begin{proof}
Following the argument given for Theorem~\ref{thm1}, and using 
$L^2$ duality, the functional to estimate for $h\in L^2 (\real^n)$ is 
$$\int_{\real^n\times\real^n\times \real^{mn}}\mkern-24mu 
 h(x) |x|^{-\beta/2} 
\prod_{k=1}^m \Big[ |y_k|^{-\rho_k} \Big( |y_k-x|\, |y_k-w|\Big)^{-(n-\alpha_k/2)}\Big] 
|w|^{-\beta/2} h(w) \,dx\,dw\,dy$$
Observe that either by the nature of the Stein-Weiss kernel or by applying 
the Brascamp-Lieb-Luttinger rearrangement theorem \cite{BLL},
 the function $h$ can be 
taken to be radial. 
Let $y_k = |y_k|\xi_k$, $x= |x|\eta_1$, $ w = |w|\eta_2$. 
By transferring the analysis first to the multiplicative group $\real_+$ and then to 
the real line, the functional integral above is equivalent to the form 
$$\int_{\real\times\real} g(x) H(x-y) g(y)\,dx\,dy$$
where $g(x) = h(e^x) e^{nx/2}$ and 
$$H(x) = \int_{S^{n-1}\times S^{n-1}} \prod_k B_k (x,n_1,n_2)\,d\eta_1 \,d\eta_2$$
with 
\begin{equation*}
\begin{split}
&B_k (x,\eta_1,\eta_2) \\
\noalign{\vskip6pt}
&\quad 
= \int_{\real \times S^{n-1}}\mkern-40mu 
 e^{-(\rho_k-\alpha_k)t} 
\left[ \Big( \cosh \Big(\frac{x}2 -t\Big) - \xi\cdot\eta_1\Big)
\Big( \cosh \Big( \frac{x}2 +t\Big)-\xi \cdot\eta_2\Big)\right]^{-(n/2 - \alpha_k/4)}
\mkern-24mu dt\,d\xi
\end{split}
\end{equation*}
Differentials on $S^{n-1}$ correspond to standard surface measure. 
Then by Young's inequality, the constant for bounding the above form 
in terms in $(\|h\|_2)^2$ is 
$$\frac1{\sigma (S^{n-1})} \int_\real H(x)\,dx$$
which provides the sharp value for $C$ in Theorem~\ref{iteratedSW}.
\end{proof}

Bessel potentials provide a framework that extends the result of 
Theorem~\ref{thm2} in the sense of Gagliardo-Nirenberg estimates.
Define 
\begin{gather*}
\Lambda_* (f;\alpha_1,\ldots,\alpha_m) 
= \int_{\real^n\times\cdots \times\real^n} 
\Big| \prod_{k=1}^m (1-\Delta_k)^{\alpha_k/4}\, f\Big|^2 
dx_1 ,\ldots,dx_m\\
\noalign{\vskip6pt}
= \int_{\real^n\times\cdots\times\real^n} \prod_{k=1}^m 
(1+|\xi_k|^2)^{\alpha_k/2}  |\,\widehat f\,|^2 \, 
d\xi_1,\ldots,d\xi_m
\end{gather*}
Bessel potentials are defined by 
\begin{equation*}
G_\alpha (x) = \frac1{(4\pi)^{\alpha/2} \Gamma (\alpha/2)} 
\int_0^\infty e^{-\pi |x|^2/\delta} \
e^{-\delta/4\pi}\ 
\delta^{-(n-\alpha)/2}\ \frac1{\delta}\, d\delta
\end{equation*}
with the properties: 
\begin{equation*}
G_\alpha (x) \ge 0\ ,\quad 
G_\alpha \in L^1 (\real^n)\ ,\quad 
G_\alpha = \F \left[ (1+4\pi^2 |\xi|^2)^{-\alpha/2}\right]\ ,\qquad 
\alpha >0
\end{equation*}
and for $0<\alpha <n$
\begin{equation*}
G_\alpha (x) = \pi^{-n/2}\ 2^{-\alpha}\ 
\Gamma\Big(\frac{n-\alpha}2\Big) \Big/ \Gamma \Big(\frac{\alpha}2\Big)\  
\ |x|^{-n+\alpha} +o (|x|^{-n+\alpha})\ \text{ as }\ |x|\to 0
\end{equation*}
and 
\begin{gather*}
G_n (x) \simeq - \left[ (4\pi)^{n/2} \Gamma (n/2)\right]^{-1} \ln |x|^2
\ \text{ as }\ |x|\to0\\
|G_\beta (x)| \le \int_{\real^n} (1+4\pi^2 |\xi|^2)^{-\beta/2} \, d\xi
\ \text{ for }\ \beta >n
\end{gather*}
and for $\alpha >0$
$$G_\alpha (x) = O(e^{-\ep |x|}) \ \text{ as }\ |x| \to\infty
\ \text{ for some }\ \ep >0$$
(see Stein \cite{Stein70}, page 132).

\begin{thm}[Gagliardo-Nirenberg inequality]\label{thm4}
For $f\in \S(\real^{mn})$, $0<\alpha_k < n$, $\alpha = \sum \alpha_k$, 
$k= 1,\ldots,m$,
$\frac1q +\frac1p =1$ with $2\le q\le 2n/(mn-\alpha)$ and $mn-\alpha <n$
\begin{gather}
\bigg[\int_{\real^n} |f(\,\underbrace{x,\ldots,x}_{m\text{ slots}}\,)|^q\,dx
\bigg]^{2/q} \le C_{\alpha,q} \Lambda_* (f;\alpha_1,\ldots,\alpha_m) 
\label{eq4}\\
\noalign{\vskip6pt}
\int_{\real^n\times\real^n} h(x) \prod_{k=1}^m G_{\alpha_k} 
(x-w) h(w)\,dx\,dw 
\le C_{\alpha,q} \Big( \|h\|_{L^p(\real^n)}\Big)^2\label{eq5}
\end{gather}
\end{thm}

\begin{proof}
Apply the method used for Theorem~\ref{thm2}.
For $q$ below the critical index, use Young's inequality to obtain 
\eqref{eq5}. 
At the critical index $q_* = 2n/(mn-\alpha)$ use the asymptotic behavior 
of the Bessel potential together with the Hardy-Littlewood-Sobolev 
inequality. 
\end{proof}

\begin{rem}
This argument gives a sharp estimate for the case $p= q=2$.
Then the sharp constant is given by 
\begin{equation*}
C_{\alpha,2} = \int_{\real^n} \prod_{k=1}^m G_{\alpha_k} (x)\,dx\ .
\end{equation*}
For $m=2$ this constant is $(4\pi)^{-n/2}\ \Gamma((\alpha-n)/2)$.
For $m=1$ this theorem reduces to a Gagliardo-Nirenberg inequality 
for the fractional Laplacian if $\alpha > 2n/q$.
Observe that there exists a constant $C$ so that 
\begin{equation*} 
\prod_{k=1}^m (1+|\xi_k|^2)^{\alpha_k/2} \le 
C\bigg[ 1 + \prod_{k=1}^m |\xi_k|^{\alpha_k}\bigg]
\end{equation*}
then 
\begin{equation*} 
\Lambda_* (f;\alpha_1,\ldots,\alpha_m) 
\le C\bigg[ \int_{\real^{mn}} |f(x_1,\ldots,x_m)|^2\,dx 
+ \Lambda (f;\alpha_1 ,\ldots, \alpha_m)\bigg]
\end{equation*}
Using a variational  argument, this Corollary is obtained from 
Theorem~\ref{thm4}.
\end{rem}

\begin{cor}
For $f\in \S (\real^{mn})$, $0< \alpha_k < n$, $\alpha = \sum\alpha_k$, 
$2\le q< 2n/(mn-\alpha)$, $mn- \alpha <n$ and 
$\theta = (mn- \frac{2n}q)/\alpha$. 
\begin{equation}\label{eq6}
\biggl[ \int_{\real^n} |f(\,\underbrace{x,\ldots,x}_{m\text{ slots}}
\,)|^q \,dx\bigg]^{2/q} 
\le D_{\alpha,q} \bigg[ \int_{\real^{mn}} |f(x_1,\ldots,x_m)|^2\,dx
\bigg]^{1-\theta} 
\bigg[ \Lambda (f;\alpha_1,\ldots, \alpha_m)\bigg]^\theta 
\end{equation}
\end{cor}

\noindent 
The case $q=2$ is included here. 
The parameter $\theta$ is restricted: $1- \frac1n < \theta <1$. 
It is not tractable to calculate sharp values for this constant.

By allowing values of the fractional powers of $(1-\Delta)$ to increase 
to the dimension of the space and above, the inequality in Theorem~\ref{thm4} 
can be extended. 
Define for two multi-indices of positive numbers 
\begin{gather*}
\bar \alpha = (\alpha_1,\ldots,\alpha_{m_1}) \ \text{ and }\ 
\bar \beta = (\beta_{m_1+1},\ldots, \beta_{m_1+m_2})\\
\noalign{\vskip6pt}
0<\alpha_k <n\ ,\ \ n\le \beta_\ell\ ,\ \  
k=1,\ldots,m_1\ ,\ \ \ell = m_1 +1,\ldots,m_1+m_2)\ ,\ \ m= m_1 +m_2
\end{gather*}
\begin{align*}
\Lambda_{\#} (f;\bar\alpha,\bar\beta) 
& = \int_{\real^n\times\cdots\times\real^n} 
\Big|\prod_{k=1}^{m_1} (1-\Delta_k/4\pi^2)^{\alpha_k/4} 
\mkern-16mu 
\prod_{\ell = m_1+1}^m \mkern-16mu
(1-\Delta_\ell /4\pi^2)^{\beta_k/4} f\Big|^2\ 
dx_1\ldots dx_m\\
\noalign{\vskip6pt}
& =\int_{\real^n\times\cdots\times\real^n} 
\prod_{k=1}^{m_1} (1+|\xi_k|^2)^{\alpha_k/2} 
\prod_{\ell = m_1+1}^m (1+|\xi_\ell|^2)^{\beta_\ell/2} |\hat f|^2\ 
d\xi_1\ldots d\xi_m
\end{align*}

\begin{thm}
[Gagliardo-Nirenberg inequality] 
\label{thm5}
For $f\in \S (\real^{mn})$, $0 < \alpha_k<n$, 
$\alpha = \sum \alpha_k$, $k=1,\ldots,m_1$, 
$\frac1q +\frac1p =1$ with $2\le q\le 2n/ (m_1n-\alpha)$ and 
$m_1n - \alpha <n$
\begin{gather} 
\Big[ \int_{\real^n} |f(\underbrace{x,\ldots,x}_{\text{$m$ slots}})|^q\,dx 
\Big]^{2/q} 
\le C_{\bar\alpha,\bar\beta,q} \Lambda_{\#} (f;\bar\alpha,\bar\beta)
\label{eq7}\\
\noalign{\vskip6pt}
\int_{\real^n\times\real^n} h(x) \prod_{k=1}^{m_1} G_{\alpha_k} (x-w) 
\prod_{\ell=m_1+1}^m G_{\beta_\ell} (x-w) h(w)\,dx\, dw 
\le C_{\bar\alpha,\bar\beta,q} 
(\|h\|_{L^p (\real^n)})^2\label{eq8}
\end{gather}
For $p=q=2$
\begin{equation}\label{eq9}
C_{\bar\alpha,\bar\beta,2} 
= \int_{\real^n} \prod_{k=1}^{m_1} G_{\alpha_k} (x) 
\prod_{\ell=m_1+1}^m G_{\beta_\ell} (x)\,dx
\end{equation}
\end{thm}

\begin{proof}
Apply Theorem~\ref{thm4} together with asymptotic estimates for the Bessel 
potentials for large and small values of $|x|$.
If all fractional powers are larger than the dimension~$n$, then the 
allowed range of $p$ extends to $1\le p\le 2$ for estimate \eqref{eq8}
with $1< p\le 2$ if some values of the fractional power are equal to $n$.
\end{proof}

\begin{rem}
In \cite{CP} Chen and Pavlovic ``introduce a generalization of Sobolev and 
Gagliardo-Nirenberg inequalities on the level of marginal density matrices,''
and their stated result corresponds to the case $q=2$ with uniform $\alpha_k$'s 
for inequality \eqref{eq7} 
above though the methods are entirely different from those used here. 
Their work is related to obtaining a~priori energy bounds for solutions to the 
Gross-Pitaevskii hierarchy. 
But there is possible confusion between their notation and that used by 
Klainerman and Machedon \cite{KM} 
(see also \cite{ESY-CPAM}, \cite{ESY-2007} and \cite{Kirk})
as they use $\langle\nabla\rangle = \sqrt{1-\Delta}$.  
To clarify potential issues, it is not possible, even for multivariable functions 
invariant under the symmetric group, to have global homogeneous Sobolev 
inequalities of trace type for the index $q=2$ in contrast to the special case of 
product functions. 
Such a result in the general case would force integrability for Riesz potentials
and on a conceptual level would ``break'' the uncertainty principle.
This phenomena is similar to the limitation discussed earlier in the context 
of the Kenig-Stein theorem for fractional integration. 
The arguments described here clearly illustrate examples where multilinear 
structure can not be reduced to the case of product functions or simple iterative 
processes though the proofs use quadratic functional integration to simplify 
the calculation where the spirit is similar to
application of a Hilbert-Schmidt norm or the Plancherel theorem.
\end{rem}

Extension of the Hardy-Littlewood-Sobolev inequality on $\real^n$ to 
include the multilinear embedding estimates described here suggests that 
analogous results should hold for the sphere $S^n$.
Following the development outlined in \cite{Beckner-PUP95}  (see page~62):

\begin{Thm}[Hardy-Littlewood-Sobolev inequality on $S^n$]
Let $F$ be a smooth function on $S^n$ with corresponding expansion in 
spherical harmonics, $F = \sum Y_k$; $\alpha = n - 2n/q$ for $q>2$ 
and define 
$$B = \left[ -\Delta +\Big(\frac{n-1}2\Big)^2\right]^{1/2}\ ,\quad 
D_\alpha = \frac{\Gamma (B+ (1+\alpha)/2)}{\Gamma (B+ (1-\alpha)/2)}\ ,$$
observe 
$$D_\alpha Y_k = \frac{\Gamma (\frac{n}{q'} +k)}{\Gamma(\frac{n}q +k)}\ Y_k\ .$$
Then 
\begin{align}
\left[\|F\|_{L^q(S^n)}\right]^2 
&\le \sum_{k=0}^\infty\ 
\frac{\Gamma(\frac{n}q)\ \Gamma(\frac{n}{q'} +k)}
{\Gamma(\frac{n}{q'}) \ \Gamma(\frac{n}q +k)}
\int_{S^n} |Y_k|^2\,d\xi \label{eq10}\\
\noalign{\vskip6pt}
\left[ \|F\|_{L^q(S^n)}\right]^2
&\le \frac{\Gamma(\frac{n-\alpha}2)}{\Gamma (\frac{n+\alpha}2)} 
\int_{S^n} F(D_\alpha F)\,d\xi \label{eq11}
\end{align}
where $d\xi$ denotes normalized surface measure on $S^n$.
\end{Thm}

\noindent 
This theorem is the sharp Hardy-Littlewood-Sobolev inequality for the 
$n$-dimensional sphere as obtained by Lieb \cite{Lieb} in terms of 
fractional integrals. 
The representation using spherical harmonics was given by 
Beckner in \cite{Beckner-Sobolev}. 

\begin{thm}[Hardy-Littlewood-Sobolev inequality]
\label{thm6}
For $F$ in the Schwartz class formed over $m$ copies of $S^n$, and 
$mn-\alpha = 2n/q$ with $0<\alpha_k <n$, $\alpha = \sum \alpha_k$ and 
$(m-1)n < \alpha < mn$. 
Let 
$$\Lambda_S (F,\alpha_1,\ldots,\alpha_m)
\equiv \prod_{k=1}^m \frac{\Gamma(\frac{n-\alpha_k}2)}
{\Gamma (\frac{n+\alpha_k}2)} 
\int_{S^n\times \cdots \times S^n} 
F\Big( \prod_k D_{\alpha_k} F\Big) \,d\xi_1,\ldots, d\xi_m$$
where $D_{\alpha_k}$ acts on the $k^{th}$ coordinate. 
Then 
\begin{gather}
\bigg[ \int_{S^n} |F(\underbrace{\xi,\ldots, \xi}_{\text{$m$ slots}}) |^q 
\,d\xi\bigg]^{2/q} 
\le F_{\alpha,S} \Lambda_S (F,\alpha_1,\ldots,\alpha_m)\label{eq12}\\
\noalign{\vskip6pt}
F_{\alpha,S} = \left[ \frac{\Gamma (\frac{n}q)}{\Gamma (n)}\right]^{m-1}
\frac{\Gamma(\frac{n}2 - \frac{n}q)}{\Gamma (\frac{n}p)}
\prod_k \frac{\Gamma (\frac{n+\alpha_k}2)}{\Gamma(\frac{\alpha_k}2)}\notag
\end{gather}
\end{thm}

\begin{proof}
Let $T_{\alpha_k} = [\Gamma (\frac{n-\alpha_k}2)/\Gamma(\frac{n+\alpha_k}2)]
D_{\alpha_k}$; then $T_{\alpha_k}$ is a positive-definite self-adjoint 
invertible operator and $T_{\alpha_k}^{-1}$ can be realized as a 
fractional integral operator on $S^n$:
$$(T_{\alpha_k}^{-1} G)(\xi) 
= 2^{n-\alpha_k} \frac{\Gamma(\frac{n}2)}{\Gamma(n)}\ 
\frac{\Gamma(\frac{n+\alpha_k}2)}{\Gamma(\frac{\alpha_k}2)}
\int_{S^n} |\xi-\eta|^{-(n-\alpha_k)} G(\eta)\,d\eta\ .$$
Inequality \eqref{eq12} is equivalent to the inequality 
$$\bigg[\int_{S^n}\Big| \Big(\prod_k T_{\alpha_k}^{-1/2} F\Big)(\xi)
\Big|^q\,d\xi\bigg]^{2/q}
\le F_{\alpha,S} \int_{S^n\times\cdots\times S^n} 
|F(\xi_1,\ldots,\xi_m)|^2\,d\xi_1,\ldots,d\xi_m$$
and by duality this is equivalent to 
$$\int_{S^n\times\cdots\times S^n} \Big|\Big(\prod_k T_{\alpha_k}^{-1/2}G\Big)
(\xi_1,\ldots,\xi_m)\Big|^2\, d\xi_1,\ldots,\xi_m 
\le F_{\alpha,S} \bigg[\int_{S^n} |G(\xi)|^p\,d\xi\bigg]^{2/p}$$
where $\frac1p +\frac1q =1$, $1<p<2$ and $mn-\alpha = 2n/q$.
As with earlier calculations, the left-hand side becomes
$$\int_{S^n\times S^n} G(\eta) \prod_k T_{\alpha_k}^{-1} (\xi,\eta) G(\xi)
\,d\xi\,d\eta 
\le F_{\alpha,S} \bigg[ \int_{S^n} |G(\xi)|^p\,d\xi \bigg]^{2/p}$$
\begin{align*}
\prod_k T_{\alpha_k}^{-1} (\xi,\eta) 
& = \prod_k2^{n-\alpha_k} \ \frac{\Gamma(\frac{n}2)}{\Gamma(n)}\ 
\frac{\Gamma(\frac{n+\alpha_k}2)}{\Gamma(\frac{\alpha_k}2)}\ 
|\xi-\eta|^{-(n-\alpha_k)}\\
\noalign{\vskip6pt}
& = 2^{2n/q} \left[\frac{\Gamma(\frac{n}2)}{\Gamma (n)}\right]^m 
\prod_k \frac{\Gamma(\frac{n+\alpha_k}2)}{\Gamma(\frac{\alpha_k}2)}\ 
|\xi-\eta|^{-2n/q}\ .
\end{align*}
Re-writing the previous inequality 
\begin{equation}\label{eq13}
A_\alpha \int_{S^n\times S^n} G(\eta)|\xi-\eta|^{-2n/q} G(\xi)\,d\xi\,d\eta
\le F_{\alpha,S} \bigg[\int_{S^n} |G(\xi)|^p\,d\xi\bigg]^{2/p} 
\end{equation}
where 
$$A_\alpha = 2^{2n/q} \left[ \frac{\Gamma(\frac{n}2)}{\Gamma(n)}\right]^m 
\prod_k \frac{\Gamma(\frac{n+\alpha_k}2)}{\Gamma(\frac{\alpha_k}2)}\ .$$
Now inequality \eqref{eq13} is the classical Hardy-Littlewood-Sobolev 
inequality on the \allowbreak 
sphere, and by comparison with the sharp constant 
(see \cite{Beckner-Annals93})
\begin{align}
F_{\alpha,S} & = A_\alpha\  2^{-2n/q} \frac{\Gamma(n)}{\Gamma(\frac{n}2)} 
\ \frac{\Gamma (\frac{n- \frac{2n}q)}2}
{\Gamma (\frac{n}p)} \notag\\
F_{\alpha,S} & = \left[\frac{\Gamma(\frac{n}2)}{\Gamma (n)}\right]^{m-1}
\ \frac{\Gamma(\frac{n}2 - \frac{n}q)}{\Gamma (\frac{n}p)} \ \prod_k
\frac{\Gamma (\frac{n+\alpha_k}2 )}{\Gamma (\frac{\alpha_k}2)}\label{eq14}
\end{align}
With the calculation of the sharp constant $F_{\alpha,S}$, the proof
of Theorem~\ref{thm6}  is complete.
\end{proof}

\begin{rem}
While the inequalities in Theorem~\ref{thm2} and \ref{thm6} directly 
depend on the conformally invariant Hardy-Littlewood-Sobolev inequality, 
these inequalities are not in themselves conformally invariant. 
This circumstance may be reflected in the lack of limiting phenomena for 
the allowed range of Lebesgue exponents. 
\end{rem}

The arguments given here, and especially for Theorem~\ref{thm6}, allow a 
fairly general formulation of the essential idea embodied in these theorems.
Suppose $\{T_k\}$ is a family of positive-definite self-adjoint invertible 
operators acting on smooth function classes on an $n$-dimensional manifold 
equipped with a suitable measure.
For simplicity, assume that $T_k^{-1}$ can be realized by action of an 
integral kernel.
Then the preceding theorems may be reformulated in the following form:

\begin{thm}[multilinear trace]\label{thm7}
For $f\in \S (M\times\cdots\times M)$ and some suitable $q\ge2$ with 
$\frac1p +\frac1q=1$ and $T_k$ acting on the $k^{th}$ variable
\begin{equation}\label{eq15}
\bigg[\int_M |f(\underbrace{x,\ldots,x}_{\text{$m$ slots}})|^q\,dx\bigg]^{2/q}
\le E_T \int_{M\times\cdots\times M} 
f\Big(\prod_k T_k\, f\Big)\,dx_1,\ldots,dx_m\ .
\end{equation}
$E_T$ is the optimal constant for the inequality 
\begin{equation}\label{eq16}
\int_{M\times M} h(w)\prod_k T_k^{-1} (x,w) h(x)\,dx\,dw
\le E_T\bigg[\int_M |h|^p\,dx\bigg]^{2/p} \ .
\end{equation}
\end{thm}

\begin{proof}
Just follow the steps in the previous argument.
\end{proof}

\begin{rem} 
In the context of the conformally invariant structure discussed in the 
author's papers \cite{Beckner-PUP95} and \cite{Beckner-JFAA}, and 
more recent treatments of the fractional Laplacian (\cite{CS}, \cite{Chang-Gonz}),
it is natural to consider the multilinear trace 
estimate in the setting where $M$ is a hyperbolic manifold.
Broader questions can be addressed in that context, and they will be 
treated in a forthcoming paper \cite{Beckner-Lie}. 
\end{rem}
\section*{Acknowledgements}

I would like to thank Aynur Bulut for drawing my attention to the 
problems discussed here, 
Natasa Pavlovic and Thomas Chen for conversations on the Gross-Pitaevskii 
hierarchy, the referee for recognizing some similar features with the 
Kenig-Stein paper, and Eli Stein for his helpful editorial suggestions.
The first version of this
paper was written while visiting the Centro di Ricerca Matematica
Ennio De Giorgi in Pisa. 
The mathematical environment was lively and stimulating, and I appreciate 
the warm hospitality of Fulvio Ricci.


\end{document}